\documentclass[12pt]{amsart}
\usepackage{amsmath,amssymb,amsthm}
\usepackage{enumitem}

\usepackage{graphicx}
\usepackage{amsfonts}
\usepackage{bigints}
\usepackage{amsfonts}
\usepackage{mathptmx}  
\usepackage{hyperref}

\usepackage{epstopdf}

\theoremstyle{plain}
\newtheorem{theorem}{Theorem}[section]

\newtheorem{corollary}[theorem]{Corollary}
\newtheorem{lemma}[theorem]{Lemma}
\theoremstyle{definition}

\newtheorem{remark}[theorem]{Remark}

\theoremstyle{Conjecture}

\makeatletter
\@addtoreset{equation}{section}
\makeatother

\makeatletter
\newcommand{\Spvek}[2][r]{%
	\gdef\@VORNE{1}
	\left(\hskip-\arraycolsep%
	\begin{array}{#1}\vekSp@lten{#2}\end{array}%
	\hskip-\arraycolsep\right)}

\def\vekSp@lten#1{\xvekSp@lten#1;vekL@stLine;}
\def\vekL@stLine{vekL@stLine}
\def\xvekSp@lten#1;{\def\temp{#1}%
	\ifx\temp\vekL@stLine
	\else
	\ifnum\@VORNE=1\gdef\@VORNE{0}
	\else\@arraycr\fi%
	#1%
	\expandafter\xvekSp@lten
	\fi}
\makeatother
\begin{document}
	
	\title{Orthogonal polynomials on generalized Julia sets}
	
	
	\author{G\"{o}kalp Alpan}
	\address{Department of Mathematics, Bilkent University, 06800 Ankara, Turkey}
	\email{gokalp@fen.bilkent.edu.tr}
	\thanks{The authors are partially supported by a grant from T\"{u}bitak: 115F199.}
	\author{Alexander Goncharov}
	\address{Department of Mathematics, Bilkent University, 06800 Ankara, Turkey}
	
	\email{goncha@fen.bilkent.edu.tr}

	\date{Received: date / Accepted: date}

	\subjclass[2010]{37F10, 42C05 \and 30C85}
	\keywords{Julia sets, Parreau-Widom sets, orthogonal polynomials, Jacobi matrices}
	\begin{abstract}
		We extend results by Barnsley et al. about orthogonal polynomials on Julia sets to the case of generalized Julia sets. The equilibrium measure is considered. In addition, we discuss optimal smoothness of Green's functions and Parreau-Widom criterion for a special family of real generalized Julia sets.


	\end{abstract}
		\maketitle
	\section{Introduction}
	
	Let $f$ be a rational function in $\overline{\mathbb{C}}$. Then the set of all points $z\in\overline{\mathbb{C}}$ such that the sequence of iterates $(f^n(z))_{n=1}^\infty$ is normal in the sense of Montel is called the \emph{Fatou} set of $f$. The complement of the Fatou set is called the \emph{Julia} set of $f$ and we denote it by $J_{(f)}$. We use the adjective \emph{autonomous} in order to refer to these usual Julia sets in the text.
	
	Potential theoretical tools for Julia sets of polynomials were developed in \cite{Brolin} by Hans Brolin. Orthogonal polynomials for polynomial Julia sets were considered in \cite{Barnsley1,Barnsley3}. Barnsley et al. show  how one can find recurrence coefficients when the Julia set $J_{(f)}$ corresponding to a nonlinear polynomial is real. Ma\~{n}\'{e} and Rocha, in \cite{Mane}, show that Julia sets are uniformly perfect in the sense of Pommerenke and in particular they are regular with respect to the Dirichlet problem.
	
	Let $(f_n)$ be a sequence of rational functions. Define $F_0(z):=z$ and $F_n(z)= f_n\circ F_{n-1}(z)$ for all $n\in\mathbb{N}$, recursively. The union of the points $z$ such that the sequence $(F_n (z))_{n=1}^\infty$ is normal is called the Fatou set for $(f_n)$ and the complement of the Fatou set is called the Julia set for $(f_n)$. We use the notation $J_{(f_n)}$ to denote it. These sets were introduced in \cite{Fornaess}. For a general overview we refer the reader to the paper \cite{Bruck}. For a recent discussion of Chebyshev polynomials on these sets, see \cite{alp4}.
	
	In this paper, we consider orthogonal polynomials with respect to the equilibrium measure of $J_{(f_n)}$ where $(f_n)$ is a sequence of nonlinear polynomials satisfying some mild conditions. To our knowledge, this paper is the first attempt dealing with the orthogonal polynomials in this generality although considerable work (see e.g. \cite{Barnsley1,Barnsley3,bes2}) has been done for the autonomous case and there are some results (see e.g. \cite{alpgon,yuditskii}) concerning the orthogonal polynomials on sets constructed using compositions of infinitely many polynomials. While the focus of \cite{yuditskii} is quite different than what we discuss, a particular family of sets considered in \cite{alpgon,gonc} clearly presents generalized Julia sets.
	
	In Section 2, we give background information about the properties of $J_{(f_n)}$ regarding potential theory. In Section 3, we prove that for certain degrees, orthogonal polynomials associated with the equilibrium measure of $J_{(f_n)}$ are given explicitly in terms of the compositions $F_n$. In Section 4, we show that the recurrence coefficients can be calculated provided that $J_{(f_n)}$ is real. These two results generalize Theorem 3 in \cite{Barnsley1} and Theorem 1 in \cite{Barnsley3} respectively. In addition to these results we discuss resolvent functions and a general method to construct real Julia sets. Techniques that we use here are rather different compared to those of autonomous setting. This is mostly due to the fact that, in the generalized case, Julia sets do not have complete invariance but we only have the properties given in part $(e)$ of Theorem 1.
	
	In Section 6, we consider a quadratic family of polynomials $(f_n)$ such that the set $K_1(\gamma)=J_{(f_n)}$ is a modification of the set $K(\gamma)$ from \cite{gonc}. In terms of the parameter $\gamma$ we give a criterion for the Green function
	$G_{\overline{\mathbb{C}}\setminus K_1(\gamma)}$ to be optimally smooth. In the last section, a criterion is
	presented for $K_1(\gamma)$ to be a Parreau-Widom set.

	\section{Preliminaries}
	Polynomial Julia sets are one of the most studied objects in one dimensional complex dynamics. For classical results related to potential theory, see \cite{Brolin}. For a more general exposition we refer to the monograph \cite{Milnor} and the survey \cite{Lyubich}.
	
	In this paper, we study in the more general framework of Julia sets. Clearly, Theorem \ref{orti} and Theorem \ref{Jacobi} are also valid for the autonomous Julia sets.
	
	Let the polynomials $f_n(z)=\sum_{j=0}^{d_n}a_{n,j}\cdot z^j$ be given where $d_n\geq 2$ and $a_{n,d_n}\neq 0$ for all $n\in\mathbb{N}$. Following \cite{Bruck}, we say that $(f_n)$ is a \emph{regular polynomial sequence} if for some positive real numbers $A_1,A_2,A_3$, the following properties are satisfied:
	\begin{itemize}
		\item $|a_{n,d_n}|\geq A_1$, $\forall n\in\mathbb{N}$.
		\item $|a_{n,j}|\leq A_2 |a_{n,d_n}|$ for $j=0,1,\ldots, d_n-1$ and $n\in\mathbb{N}$.
		\item $\log{|a_{n,d_n}|}\leq A_3\cdot d_n$
		for all $n\in\mathbb{N}$.
	\end{itemize}
	
	We use the notation $(f_n)\in\mathcal{R}$ if $(f_n)$ is a regular polynomial sequence. We remark that, for a sequence $(f_n)\in\mathcal{R}$, the degrees of polynomials need not to be the same and they do not have to be bounded above either.  Julia sets $J_{(f_n)}$ when $(f_n)\in\mathcal{R}$ were introduced and considered in \cite{Buger} and all results given in the next theorem are from Section 2 and Section 4 of the paper \cite{Bruck}. While \eqref{11} is contained in the proof of Theorem 4.2 in \cite{Bruck}, \eqref{22} follows by comparing the right parts of these two equations, using that $G_{\overline{\mathbb{C}}\setminus J_{(f_n)}}$ has a logarithmic singularity at infinity and $F_k(z)$ goes locally uniformly to $\infty$ for such $z$.
	\begin{theorem}\label{general}
		Let $(f_n)\in\mathcal{R}$. Then the following propositions hold:
		\begin{enumerate}[label={(\alph*})]
			\item The set $\mathcal{A}_{(f_n)}(\infty):=\{z\in\overline{\mathbb{C}}: \mbox{ $F_k(z)$ goes locally uniformly to $\infty$}\}$ is an open connected set containing $\infty$. Moreover, for every $R>1$ satisfying the inequality $$A_1 R\left(1-\frac{A_2}{R-1}\right)>2,$$ the compositions $F_{n}(z)$ goes locally uniformly to infinity whenever $z\in\triangle_R$ where $\triangle_R= \{z\in\overline{\mathbb{C}}: |z|>R\}.$
			\item $\displaystyle \mathcal{A}_{(f_n)}(\infty)=\cup_{k=1}^{\infty}{F_k}^{-1}(\triangle_R)$ and $f_n({\overline{{\triangle}_R}})\subset \triangle_R$ if $R>1$ satisfies the inequality given in part $(a)$. Furthermore, we have $J_{(f_n)}=\partial \mathcal{A}_{(f_n)}(\infty)$.
			\item $J_{(f_n)}$ is regular with respect to the Dirichlet problem. The Green function for the complement of the set is given by 
			\begin{equation}\label{11}
				G_{\overline{\mathbb{C}}\setminus J_{(f_n)}}(z) =
				\left\{
				\begin{array}{ll}
					\lim_{k\rightarrow\infty} \frac{1}{d_1\cdots d_k} \log{|F_k(z)|} & \mbox{if } z\in \mathcal{A}_{(f_n)}(\infty),\\ 0 & \mbox{otherwise. }
				\end{array}
				\right. 
			\end{equation}
			Moreover, 
			\begin{equation}\label{22}
				G_{\overline{\mathbb{C}}\setminus J_{(f_n)}}(z) =\lim_{k\rightarrow\infty} \frac{1}{d_1\cdots d_k} G_{\overline{\mathbb{C}}\setminus J_{(f_n)}}{(F_k(z))}
			\end{equation}
			where $z\in\mathcal{A}_{(f_n)}(\infty)$. In both \eqref{11} and \eqref{22}, limits hold locally uniformly in $\mathcal{A}_{(f_n)}(\infty).$
			
			\item The logarithmic capacity of the compact set $J_{(f_n)}$ is given by the expression $$\displaystyle \mathrm{Cap}(J_{(f_n)})= \exp{\left(-\lim_{k\rightarrow \infty}\sum_{j=1}^{k}\frac{\log{|a_{j,d_j}|}}{d_1\cdots d_j}\right)}.$$
			
			\item $F_k^{-1}(F_k(J_{(f_n)}))=J_{(f_n)}$ and $J_{(f_n)}= F_k^{-1}(J_{(f_{k+n})})$ for all $k\in\mathbb{N}$. Here we use the notation $(f_{k+n})=(f_{k+1},f_{k+2},f_{k+3}, \ldots).$
		\end{enumerate}
	\end{theorem}
	
	We have to note that for the sequences $(f_n)\in\mathcal{R}$ satisfying the additional condition $d_n=d$ for some $d\geq 2$, there is a nice theory concerning topological properties of Julia sets. For details, see \cite{comerford,Urbanski}.
	
	Before going any further, we want to mention the results from \cite{Barnsley1} and \cite{Barnsley3} concerning orthogonal polynomials for the autonomous Julia sets. Let $f(z)= z^n+ k_1 z^{n-1}+\ldots+ k_n$ be a nonlinear monic polynomial of degree $n$ and let $P_j$ denote the $j$-th monic orthogonal polynomial associated to the equilibrium measure of $J_{(f)}$. Then we have,
	
	\begin{enumerate}[label={(\alph*})]
		\item $P_1(z)= z+{k_1}/n$.
		\item $P_{l n} (z)=P_l(f(z))$, for $l=0,1,\ldots$
		\item $P_{n^l}(z)=f^l(z)+{k_1}/n$ for $l=0,1,\ldots$, where $f^l$ is the $l$-th iteration of the function $f$.
	\end{enumerate}
	
	In Theorem \ref{orti}, we recover parts $(a)$ and $(c)$ in a more general setting. Even without having the analogous equations to part $(b)$, recurrence coefficients appear as the outcome of Theorem \ref{Jacobi}.
	
	Throughout the whole article when we say that $(f_n)\in\mathcal{R}$ then the sequences $(d_n)$, $(a_{n,j})$, $(A_i)_{i=1}^3$ will be used just as in the definition given in the beginning of this section and $F_n(z)$ will stand for $f_n\circ \ldots \circ f_1(z)$. Thus $F_n$ is a polynomial  with the leading coefficient $(a_{1,d_1})^{d_2\cdots d_l}(a_{2,d_2})^{d_3\cdots d_l}\cdots a_{l,d_l}$ of degree $d_1\cdots d_{n}.$  For a compact non-polar set $K$, we denote the Green function of $\Omega$ with pole at infinity by $G_{\overline{\mathbb{C}}\setminus K}$ where $\Omega$ is the connected component of $\overline{\mathbb{C}}\setminus K$ containing $\infty$. We use $\mu_K$ to denote the equilibrium measure of $K$. Convergence of measures is considered in weak-star topology. In addition, we consider and count multiple roots of a polynomial separately.
	\section{Orthogonal polynomials}
	We begin with a lemma due to Brolin \cite{Brolin}.
	\begin{lemma}\label{Brolin}
		Let $K$ and $L$ be two non-polar compact subsets of $\mathbb{C}$ such that $K\subset L$. Let $(\mu_n)_{n=1}^\infty$ be a sequence of probability measures supported on $L$ that converges to a measure $\mu$ supported on $K$. Suppose that the following two conditions hold where $U_n(z)$ stands for the logarithmic potential for the measure $\mu_n$ and $V_K$ is the Robin constant for $K$:
		\begin{enumerate}[label={(\alph*})]
			\item $\displaystyle\liminf_{n\rightarrow\infty}U_n(z)\geq V_K$ on $K$.
			\item $ \mathrm{supp}{(\mu_K)}=K$.
		\end{enumerate}
		Then $\mu=\mu_K$  .
	\end{lemma}
	
	Let $(f_n)\in\mathcal{R}$. Then, by the fundamental theorem of algebra (FTA), $F_k(z)-a=0$ has $d_1\cdots d_k$ solutions counting multiplicities. For a fiven $k$, let us define the normalized counting measure as $\nu_{k}^a=\frac{1}{d_1\cdots d_k}\sum_{l=1}^{d_1\cdots d_k}  \delta_{z_l}$ where $z_1, \ldots, z_{d_1\cdots d_k}$ are the roots of $F_k(z)-a$. In \cite{Brolin} and later on in \cite{Bruck1}, it is shown that $\nu_{k}^a\rightarrow \mu_{J_{(f_n)}}$ for a proper $a$ where in the first article $f_n=f$ with a monic nonlinear polynomial $f$ and in the second one $f_n(z)=z^2+c_n$. Our technique used below is the same in essence with the proofs in \cite{Brolin,Bruck1}. Due to some minor changes and for the convenience of the reader, we include the proof of the theorem.
	
	\begin{theorem}\label{weak} Let $(f_n)\in\mathcal{R}$. Then for $a\in\mathbb{C}\setminus \overline{\mathbb{D}}$ satisfying the condition
		\begin{equation}\label{eq1}
			|a|A_1\left(1-\frac{A_2}{|a|-1}\right)>2,
		\end{equation}
		we have $\nu_{k}^a\rightarrow \mu_{J_{(f_n)}}$.
	\end{theorem}
	\begin{proof} Choose a number $a\in\mathbb{C}\setminus \overline{\mathbb{D}}$ satisfying \eqref{eq1}. Let $K:=J_{(f_n)}$ and $L:=\{z\in\mathbb{C}: |z|\leq |a| \}$. Then, by part $(b)$ of Theorem \ref{general}, $K\subsetneq L$. Moreover, since $K$ is regular with respect to the Dirichlet problem and $K$ is equal to the boundary of the component of $\overline{\mathbb{C}}\setminus K$ that contains $\infty$, we have (see e.g. Theorem 4.2.3. of \cite{Ransford}) that $\mathrm{supp}{(\mu_K)}=K$.
		
		Observe that, ${F_k}^{-1}(a)\cap\mathcal{A}_{(f_n)}(\infty)$ is contained in $L$ for all $k\in\mathbb{N}$ by part $(b)$ of Theorem \ref{general}. Thus, $(\nu_k^{a})_{k=1}^{\infty}$ has a convergent subsequence $({\nu_{k_l}^{a}})_{l=1}^\infty$ by Helly's selection principle (see e.g. Theorem 0.1.3. in \cite{saff}). Let us denote the limit by $\mu$. The set $\cup{F_k}^{-1}(a)$ can not accumulate to a point $z$ in $\mathcal{A}_{(f_n)}(\infty)$, since this would contradict with the fact that $F_k(z)$ goes locally uniformly to $\infty$  by part $(a)$ of Theorem \ref{general}. Thus, $\mathrm{supp}(\mu)\subset \partial \mathcal{A}_{(f_n)}(\infty)=K$.
		
		Now, we want to show that $\displaystyle\liminf_{l\rightarrow\infty}U_{k_l}(z)\geq V_K$ for all $z\in K$. Let $z\in K$ where $U_k$ denote the logarithmic potential for $\nu_k^{a}$. We have $$|F_{k_l}(z)-a|= |(a_{1,d_1})^{d_2\cdots d_{k_l}}||(a_{2,d_2})^{d_3\cdots d_{k_l}}|\cdots |a_{k_l,d_{k_l}}|\prod_{j=1}^{d_1\cdots d_{k_l}} |z-z_{j,k_l}|,$$ for some $z_{j,k_l}\in L$. Thus,
		\begin{equation}\label{eq2} U_{k_l}(z)=\frac{\sum_{j=1}^{d_1\cdots d_{k_l}}\log{|z-z_{j,k_l}|}}{-d_1\cdots d_{k_l}}=\sum_{j=1}^{d_1\cdots d_{k_l}}\frac{\log{|a_{j,d_j}|}}{d_1\cdots d_j}-\frac{\log{|F_{k_l}(z)-a|}}{d_1\cdots d_{k_l}}.
		\end{equation}
		Using part $(d)$ of Theorem \ref{general} and the fact that $|F_k(z)|\leq |a|$ for $z\in K$, we see that the following inequality follows from \eqref{eq2}:  $$\displaystyle \liminf_{l\rightarrow\infty} U_{k_l}(z)\geq \liminf_{l\rightarrow\infty} \left(\sum_{j=1}^{d_1\cdots d_{k_l}}\frac{\log{|a_{j,d_j}|}}{d_1\cdots d_j}-\frac{\log{|2a|}}{d_1\cdots d_{k_l}}\right)\geq V_K.$$
		Hence, by Lemma \ref{Brolin}, we have $\nu_{k_l}^a\rightarrow \mu_K$. Since $(\nu_{k_l}^a)$ is an arbitrary convergent subsequence,  $\nu_{k}^a\rightarrow \mu_K$ also holds.
	\end{proof}
	
	In the next theorem, we use algebraic properties of polynomials as well as analytic properties of the corresponding Julia sets. Let $f(z)= a_n z^n+ a_{n-1}z^{n-1}+\ldots a_0$ be a nonlinear polynomial of degree $n$ and let $z_1, z_2, \ldots, z_n$ be roots of $f$ counting multiplicities. Then, for $k=1,2,\ldots n-1$ we have the following Newton's identities:
	\begin{equation}\label{newton}
		s_k(f(z))+\frac{a_{n-1}}{a_n}s_{k-1}(f(z))+
		\ldots+\frac{a_{n-k+1}}{a_n}s_1(f(z))=-k\frac{a_{n-k}}{a_n},
	\end{equation}
	where $s_k(f(z)):=\sum_{j=1}^n (z_j)^k$.
	
	For the proof of \eqref{newton}, see \cite{Mead} among others. Note that, none of these equations include the term $a_0$. This implies that the values $(s_k)_{k=1}^{n-1}$ are invariant under translation of the function $f$, i.e.
	\begin{equation}\label{invar}
		s_k(f(z))=s_k(f(z)+c)
	\end{equation} for any $c\in\mathbb{C}$. Let $(P_j)_{j=1}^{\infty}$ denote the sequence of monic orthogonal polynomials associated to $\mu_{J_{(f_n)}}$ where $\deg{P_j}=j$. Now we are ready to prove our first main result.
	\begin{theorem}\label{orti}
		For $(f_n)\in \mathcal{R}$, we have the following identities:
		\begin{enumerate}[label={(\alph*})]
			\item $\displaystyle P_1(z)=z+ \frac{1}{d_1}\frac{a_{1, d_1-1}}{ a_{1,d_1}}.$
			\item $\displaystyle P_{d_1\cdots d_l}(z)=\frac{1}{(a_{1,d_1})^{d_2\cdots d_l}(a_{2,d_2})^{d_3\cdots d_l}\cdots a_{l,d_l}}\left(F_l(z)+\frac{1}{d_{l+1}}\frac{a_{l+1,d_{l+1}-1}}{a_{l+1,d_{l+1}}}\right).$
		\end{enumerate}
	\end{theorem}
	\begin{proof}
		(a) Let $(f_n)\in\mathcal{R}$ be given and $a\in\mathbb{C}\setminus \overline{\mathbb{D}}$ satisfy \eqref{eq1}. Fix an integer $m$ greater than $1$. By FTA, The solutions of the equation $F_m(z)=a$ satisfy an equation of the form $$\left(F_{m-1}(z)-\beta_{m-1}^1\right)\dots(F_{m-1}(z)-\beta_{m-1}^{d_m})=0,$$
		where $\beta_{m-1}^1,\ldots,\beta_{m-1}^{d_m}\in\mathbb{C}$. The $d_1\cdots d_{m-1}$ roots of the equation $F_{m-1}-\beta_{m-1}^j=0$ are the solutions of an equation
		$$(F_{m-2}(z)-\beta_{m-2}^{1,j})\dots (F_{m-2}(z)-\beta_{m-2}^{d_{m-1},j})=0,$$ with some $\beta_{m-2}^{1,j},\dots, \beta_{m-2}^{d_{m-1},j}$. Continuing this way, the points satisfying the equation $F_m(z)=a$ can be grouped into $d_2\cdots d_m$ parts of size $d_1$ such that each part consists of the roots of an equation $$f_1(z)-\beta_1^{j}=0,$$ for $j\in\{1,\ldots, d_2\cdots d_m\}$ and $\beta_1^j\in\mathbb{C}$. If  for each $j$, we denote the normalized counting measure on the roots of $f_1(z)-\beta_1^{j}$ by $\lambda_j$, then $$\nu_m^a= \frac{1}{d_2\cdots d_m}\sum_{j=1}^{d_2\cdots d_m} \lambda_j.$$ Hence, by \eqref{newton} and \eqref{invar},
		\begin{equation*}
			\int z\, d\nu_m^a =\frac{1}{d_2\cdots d_m}\sum_{j=1}^{d_2\cdots d_m}\int z\,d\lambda_j =
			\frac{1}{d_2\cdots d_m}\sum_{j=1}^{d_2\cdots d_m}\frac{s_1(f_1(z)-\beta_1^j)}{d_1}
		\end{equation*}
		$$
		= \frac{1}{d_1\cdots d_m}\sum_{j=1}^{d_2\cdots d_m}s_1(f_1(z)) =
		-\frac{1}{d_1}\frac{a_{1,d_1-1}}{a_{1,d_1}}.
		$$
		Since $\nu_m^a$ converges to the equilibrium measure of ${J(f_n)}$ by Theorem \ref{weak}, the result follows.
		
		(b) Let $m, l\in\mathbb{N}$ where $m>l+1$. As above, the roots of the equation $F_m(z)=a$ where $a\in\mathbb{C}\setminus \overline{\mathbb{D}}$ satisfies \eqref{eq1}, can be grouped into $d_{l+2}\cdots d_m$ parts of size $d_1\cdots d_{l+1}$ such that each part obeys an equation of the form $$F_{l+1}(z)-\beta_{l+1}^{j}=0,$$ for $j=1,2,\ldots, d_{l+2}\cdots d_m$. Recall that
		$F_{l+1}(z)=f_{l+1}(t)$ with $t= F_{l}(z).$
		
		By FTA, we have $f_{l+1}(t)-\beta_{l+1}^{j}=(t-\beta_{l}^{1,j})\cdots (t-\beta_{l}^{d_{l+1},j})$ for some $\beta_{l}^{1,j},\ldots, \beta_{l}^{d_{l+1},j}$. By \eqref{newton} and \eqref{invar}, for $k\in\{1,\ldots, d_{l+1}-1\}$ and $j,j^\prime\in\{1,\ldots, d_{l+2}\cdots d_m\}$, we have
		\begin{equation*}
			s_k(f_{l+1}(t)-\beta_{l+1}^{j}):= \sum_{r=1}^{d_{l+1}}(\beta_{l}^{r,j})^k=\sum_{r=1}^{d_{l+1}}(\beta_{l}^{r,j^\prime})^k= s_k(f_{l+1}(t)-\beta_{l+1}^{j^\prime}).
		\end{equation*}
		
		Now we can rewrite $F_{l+1}(z)-\beta_{l+1}^{j}=0$ as $(F_l(z)-\beta_{l}^{1,j})\cdots (F_l(z)-\beta_{l}^{d_{l+1},j})=0$ for $j$ as above.
		Let us denote the normalized counting measures on the roots of $F_l(z)-\beta_{l}^{r,j}=0$ by $\lambda_{r,j}$ for $r=1,\ldots,d_{l+1}$ and $j=1,\ldots,d_{l+2}\cdots d_m$. Clearly, this yields
		\begin{equation}\label{countt}
			\nu_m^a= \frac{1}{d_{l+2}\cdots d_m}\sum_{j=1}^{d_{l+2}\cdots d_m}\frac{1}{d_{l+1}}\sum_{r=1}^{d_{l+1}} \lambda_{r,j}=\frac{1}{d_{l+1}\cdots d_m}\sum_{j=1}^{d_{l+2}\cdots d_m}\sum_{r=1}^{d_{l+1}}\lambda_{r,j}.
		\end{equation}
		Thus, by using \eqref{countt}, \eqref{newton} and \eqref{invar}, we deduce that
		\begin{align*}
			\int F_l(z)\,d\nu_m^a &= \frac{1}{d_{l+1}\cdots d_m} \sum_{j=1}^{d_{l+2}\cdots d_m}\sum_{r=1}^{d_{l+1}}\int F_l(z)\,d\lambda_{r,j}\\
			&= \frac{1}{d_{l+1}\cdots d_m} \sum_{j=1}^{d_{l+2}\cdots d_m}\sum_{r=1}^{d_{l+1}}\beta_{l}^{r,j}\\
			&=\frac{1}{d_{l+1}\cdots d_m}\sum_{j=1}^{d_{l+2}\cdots d_m}s_1(f_{l+1}(t)-\beta_{l+1}^{j})\\
			&=\frac{1}{d_{l+1}\cdots d_m}\sum_{j=1}^{d_{l+2}\cdots d_m}s_1(f_{l+1}(t))\\
			&=-\frac{1}{d_{l+1}}\frac{a_{l+1,d_{l+1}-1}}{a_{l+1,d_{l+1}}}.
		\end{align*}
		To shorten notation, we write $c$ instead of $\frac{1}{d_{l+1}}\frac{a_{l+1,d_{l+1}-1}}{a_{l+1,d_{l+1}}}.$
		Thus, we have
		\begin{equation}\label{varan1}
			\int\left(F_l(z)+c\right)d\nu_m^a=0.
		\end{equation}
		
		Let us show that the integrand is orthogonal to $z^k$ with $1\leq k \leq d_1\cdots d_l-1$ as well. For the same
		$\lambda_{r,j},$ as above, we have
		$$ \int\left(F_l(z)+ c \right)\overline{z^k}\,d\lambda_{r,j} =\frac{1}{d_1\cdots d_{l}} \left(\beta_{l}^{r,j}+ c\right) \cdot \, \overline{s_k\left(F_l(z)-\beta_{l}^{r,j}\right)}.$$
		By \eqref{invar}, $ \overline{s_k\left(F_l(z)-\beta_{l}^{r,j}\right)}= \overline{s_k\left(F_l(z)\right)},$ so it does not
		depend on $r$ or $j.$ This and the representation \eqref{countt} imply that
		$$ \int\left(F_l(z)+c \right)\overline{z^k}\,d\nu_m^a= \frac{1}{d_{l+1}\cdots d_{m}}\sum_{j=1}^{d_{l+2}\cdots d_m} \sum_{r=1}^{d_{l+1}} \int\left(F_l(z)+ c \right)\overline{z^k}\,d\lambda_{r,j} $$
		$$=\frac{\overline{s_k\left(F_l(z)\right)}}{d_1\dots d_l} \int\left(F_l(z)+c \right)\,d\nu_m^a,$$
		where the last term is equal to $0$, by \eqref{varan1}. It follows that $\left(F_l(z)+c\right)\perp z^k$ for $k\leq \deg F_l-1$
		in $L^2(\mu_{J_{(f_n)}}),$ since $\nu_m^a$ converges to the equilibrium measure of ${J(f_n)}.$ This completes the proof of the theorem. 
		
	\end{proof}
	\section{Moments and resolvent functions}
	In this section we consider Julia sets that are subsets of the real line. 
	
	If $\mu$ is a probability measure which has infinite compact support  in $\mathbb{R}$, then the monic orthogonal polynomials $(P_n)_{n=1}^\infty$ satisfy a recurrence relation
	\begin{equation*}
		P_{n+1}(x)=(x-{b_{n+1}})P_n(x)-a_n^2 P_{n-1}(x),
	\end{equation*}
	assuming that $P_0=1$ and $P_{-1}=0$. If the moments $c_n=\int x^n d\mu$ are known for all $n\in\mathbb{N}_0$ then we have the formula
	\begin{equation}\label{hankel}
		p_n(x)=\frac{1}{\sqrt{D_n D_{n-1}}}\begin{vmatrix} c_0 & c_1 & \dots & c_n\\c_1 & c_2 & \dots &c_{n+1}\\ \vdots& \vdots & &\vdots\\ c_{n-1} & c_n & \dots & c_{2n-1}\\1 & x & \dots & x^n\end{vmatrix}
	\end{equation}
	where $p_n$ is the $n$-th orthonormal polynomial and $D_n$ is the determinant for the matrix $M_n$ with the entries $(M_n)_{i,j}=c_{i+j}$ for $i,j=0,1,\ldots n$. From \eqref{hankel}, one can also calculate recurrence coefficients $(a_n,b_n)_{n=1}^{\infty}$. See \cite{ase} for a detailed description of the orthogonal polynomials on the real line. In the next theorem, we show that the moments for the equilibrium measure of $J_{(f_n)}$ can be calculated recursively whenever $(f_n)\in\mathcal{R}$. Note that $c_0=1$ since the equilibrium measure is of unit mass.
	\begin{theorem}\label{Jacobi}Let $(f_n)\in\mathcal{R}$ and $l>0$ be an integer. Furthermore, let $$\frac{F_l(z)}{p_l}=z^{d_1\cdots d_l}+a_{d_1 d_2\cdots d_l-1}z^{d_1 d_2\cdots d_l-1}+\ldots+a_1 z+a_0,$$ where $p_l$ is the leading coefficient for $F_l$. Then, each moment $c_k=\int x^k d\mu_{J_{(f_n)}}$ for $k\in\{1, 2 \ldots, (d_1 d_2\cdots d_l)-1\}$ is equal to $\frac{s_k(F_l(z))}{d_1\cdots d_l}$ where $s_k(F_{l}(z))$ can be calculated recursively by Newton's identities.
	\end{theorem}
	\begin{proof}
		Let $m$ be an integer greater than $l$. Consider the roots of the equation $F_m(z)=a$ where $a\in\triangle_1$ satisfies the condition \eqref{eq1}. Then, following the proof of Theorem \ref{orti}, we can divide these roots into $d_{l+1}\cdots d_m$ parts of size $d_1\cdots d_l$ such that the nodes in each of the groups constitute the roots of an equation of the form  $$ F_l(z)-\beta^j=0,$$ for $j=1, 2, \ldots, d_{l+1}\cdots d_m$. If for each $j$ we denote the normalized counting measure on the roots of $F_l(z)-\beta^j$ by $\lambda_j$, then by \eqref{newton} and \eqref{invar}, this leads to
		\begin{align*}
			\int x^k\, d\nu_{m}^a 
			&= \frac{1}{d_{l+1}\cdots d_m}\sum_{j=1}^{d_{l+1}\cdots d_m}\int\,x^k\,d\lambda_j\\
			&=\frac{1}{d_{l+1}\cdots d_m}\sum_{j=1}^{d_{l+1}\cdots d_m}\frac{s_k(F_l(z)-\beta^j)}{d_1\cdots d_l}\\
			&=\frac{1}{d_{l+1}\cdots d_m}\sum_{j=1}^{d_{l+1}\cdots d_m}\frac{s_k(F_l(z))}{d_1\cdots d_l} =\frac{s_k(F_l(z))}{d_1\cdots d_l},
		\end{align*}
		for $k=1, 2 \ldots, (d_1 d_2\cdots d_l)-1$.  Since the weak star limit of the sequence $(d\nu_{m}^a)$ is the equilibrium measure of the Julia set by Theorem \ref{weak}, we have $\int x^k d\mu_{J_{(f_n)}}= \frac{s_k(F_l(z))}{d_1\ldots d_l}$ which concludes the proof.
	\end{proof}
	
	In Sections 3-5 of \cite{alpgon}, orthogonal polynomials and recurrence coefficients are discussed for the quadratic case. It would be interesting to obtain similar results for $\mu_{J_{(f_n)}}$ if we only assume that $(f_n)\in\mathcal{R}$ and $J_{(f_n)}\subset \mathbb{R}$.
	
	For two bounded sequences $(a_n)_{n=1}^\infty$ and $(b_n)_{n=1}^\infty$ with $a_n>0$ and $b_n\in\mathbb{R}$ for $n\in\mathbb{N}$, the associated (half-line) Jacobi operator $H: \ell ^2(\mathbb{N})\rightarrow \ell ^2(\mathbb{N})$ is given by $(Hu)_n=a_n u_{n+1}+ b_n u_n + a_{n-1} u_{n-1}$ for $u\in \ell ^2(\mathbb{N})$ and $a_0:=0$. Here, $\ell ^2(\mathbb{N})$ denotes the space of square summable sequences in $\mathbb{N}$. The spectral measure of $H$ for the cyclic vector $\delta_1=(1,0,0,\ldots)^T$ is just the one which has $a_n, b_n$ $(n=1,2\ldots)$ as the recurrence coefficients. 
	
	Let $J_{(f_n)}\subset [-M,M]$ for some $M\in\mathbb{R}$ where $(f_n)\in\mathcal{R}$. If we denote the Jacobi operator associated with $\mu_{J_{(f_n)}}$ by $H_{(f_n)}$ then the \emph{resolvent function} $R_{(f_n)}$ is defined as 
	\begin{equation*}
		R_{(f_n)}(z):= \bigintsss \frac{d\, \mu_{J_{(f_n)}}(x)}{x-z} = \langle (H_{(f_n)}-z)^{-1} \delta_1, \delta_1\rangle
	\end{equation*}
	for $z\in\mathbb{C}\setminus J_{(f_n)}$. Note that $R_{(f_n)}$ is an analytic function. If $f_n=f$ for a nonlinear polynomial $f$ for all $n\in\mathbb{N}$ then the resolvent function satisfies a functional equation:
	\begin{equation}\label{55}
		R_{(f)}(z)= \frac{f^{\prime}(z)}{\deg f} R_{(f)}(f(z)).
	\end{equation}
	
	See e.g. \cite{bes2} for a discussion of resolvent functions and operators associated with the equilibrium measure of autonomous polynomial Julia sets. It is well known that (see e.g. p. 53 in \cite{Sim3}) for $z\in\mathbb{C}\setminus \overline{D_M(0)}$ 
	\begin{equation}\label{33}
		R_{(f_n)}(z) = -\sum_{n=0}^\infty c_n z^{-(n+1)}
	\end{equation}
	where $c_n$ is the $n$-th moment for $\mu_{J_{(f_n)}}$, $D_M(0)$ is the open ball centered at $0$ with radius $M$ in $\mathbb{C}$ and the series at \eqref{33} is absolutely convergent in the corresponding domain. 
	
	We define the $\partial$ operator as 
	\begin{equation*}
		\partial= \frac{\partial_x - i \partial_y}{2}.
	\end{equation*}
	If $g$ is a harmonic function on a simply connected domain $D\subset\mathbb{C}$ then (see e.g. Theorem 1.1.2 in \cite{Ransford}) there is an analytic function $h$ on $D$ such that $g= \mathrm{Re} \,h$ holds. Moreover, we have $h^\prime (z)= 2 \partial g(z)$. Furthermore, 
	\begin{equation*}
		G_{\overline{\mathbb{C}}\setminus J_{(f_n)}}(z)= \log{(\mathrm{Cap}(J_{(f_n)})^{-1})}- U_{\mu_{J_{(f_n)}}}(z)
	\end{equation*}
	holds where $U_{\mu_{J_{(f_n)}}}$ is the logarithmic potential for $\mu_{J_{(f_n)}}$. In addition, for each $z_0\in\mathbb{C}\setminus J_{(f_n)}$, there is a $\delta>0$ and an analytic function $h$ (which may depend on $z_0$) such that (see e.g. p. 87 in \cite{finkel}) $h^\prime(z)= R_{(f_n)}(z)$ and $\mathrm{Re}\, h= U_{\mu_{J_{(f_n)}}}$ for $z\in D_\delta(z_0)$. By harmonicity of $U_{\mu_{J_{(f_n)}}}$ this implies  
	\begin{equation}\label{77}
		2\partial G_{\overline{\mathbb{C}}\setminus J_{(f_n)}}(z)= -2\partial U_{\mu_{J_{(f_n)}}}(z)= -R_{(f_n)}(z)
	\end{equation}
	
	for all $z\in\mathbb{C}\setminus J_{(f_n)}$. The next theorem follows from the discussion above. 
	
	\begin{theorem}\label{ttt}
		Let $J_{(f_n)}\subset \mathbb{R}$ provided that $(f_n)\in\mathcal{R}$. Then the following functional equation holds where the limit exists locally uniformly in $\mathbb{C}\setminus J_{(f_n)}$:
		\begin{equation}\label{88}
			R_{(f_n)}(z) = \lim_{k\rightarrow\infty} \frac{R_{(f_n)}(F_k(z)) F_k^\prime (z)}{d_1\cdots d_k}.
		\end{equation}
	\end{theorem}
	\begin{proof}
		If we apply $\partial$ to both sides of \eqref{22}, it is permitted to change the differentiation and limit since (see e.g. p. 16 in \cite{axler}) $G_{\overline{\mathbb{C}}\setminus J_{(f_n)}}$ is harmonic in $\mathcal{A}_{(f_n)}(\infty)\setminus \infty$. Note that $\mathcal{A}_{(f_n)}(\infty)\setminus \infty=\mathbb{C}\setminus J_{(f_n)}$ here since $J_{(f_n)}$ lies on $\mathbb{R}$. Hence, we have 
		\begin{equation}\label{66}
			\partial G_{\overline{\mathbb{C}}\setminus J_{(f_n)}}(z)= \lim_{k\rightarrow\infty} \frac{\partial G_{\overline{\mathbb{C}}\setminus J_{(f_n)}}(F_k(z)) F_k^\prime (z)}{d_1\cdots d_k}
		\end{equation}
		where the limit on the right side of \eqref{66} holds locally uniformly. Using \eqref{77} and \eqref{66}, we have \eqref{88} immediately.
	\end{proof}
	\begin{remark}
		Provided that $f_n=f$ for a fixed nonlinear polynomial $f$ in Theorem \ref{ttt}, \eqref{88} reduces to \eqref{55} if we put $f(z)$ instead of $z$ in both sides of \eqref{22} and follow the steps of the proof of Theorem \ref{ttt}. 
	\end{remark}
	
	\section{Construction of real Julia sets}
	Let $f$ be a nonlinear real polynomial with real and simple zeros $x_1<x_2<\ldots<x_n$ and distinct extremas $y_1<\ldots<y_{n-1}$ with $|f(y_i)|>1$ for $i=1,2, \ldots, n-1$. Then we say that $f$ is an \emph{admissible} polynomial. Note that in the literature the last condition is usually given as $|f(y_i)|\geq 1$. We list useful features of preimages of admissible polynomials.
	
	\begin{theorem}\cite{van assche}\label{inv} Let $f$ be an admissible polynomial of degree $n$. Then $$f^{-1}([-1,1])=\cup_{i=1}^n E_i$$ where $E_i$ is a closed non-degenerate interval containing exactly one root $x_i$ of $f$ for each $i$. These intervals are pairwise disjoint and $\mu_{f^{-1}([-1,1])}(E_i)=1/n$.
	\end{theorem}
	
	We say that an admissible polynomial $f$ satisfies {\it the property} $(A)$ if 
	\begin{enumerate}[label={(\alph*})]
		\item $f^{-1}([-1,1])\subset [-1,1],$
		\item $f(\{-1,1\})\subset \{-1,1\}$,
		\item $f(a)=0$ implies $f(-a)=0$.
	\end{enumerate}
	
	Clearly, $(c)$ implies that $f$ is even or odd.
	
	\begin{lemma}\label{admlem}
		Let $g_1$ and $g_2$  be admissible polynomials satisfying  $(A)$.
		Then $g_3:=g_2\circ g_1$ is also an admissible polynomial that satisfies  $(A).$
	\end{lemma}
	\begin{proof}
		Let $\deg g_k= n_k$. Moreover, let $(x_{j,1})_{j=1}^{n_1}$, $(x_{j,2})_{j=1}^{n_2}$ be the zeros and $(y_{j,1})_{j=1}^{n_1-1}$ and $(y_{j,2})_{j=1}^{n_2-1}$ be the critical points of $g_1,\,\,g_2$ respectively. Then the equation $g_3(z)=0$ implies that $g_1(z)=x_{j,2}$ for some $j\in\{1,\ldots,n_2\}$. By $(a)$ and $(b)$, the equation $g_1(z)=\beta$ has $n_1$ distinct roots for $|\beta|\leq 1$ and the sets of roots of $g_1(z)=\beta_1$ and $g_1(z)=\beta_2$ are disjoint for different $\beta_1,\beta_2\in[-1,1]$. Therefore, $g_3$ has $n_1 n_2$ distinct zeros. Similarly, $(g_3)^\prime (z)=g_2^\prime(g_1(z))g_1^\prime (z)=0$ implies $g_1^\prime (z)=0$ or $g_1(z)=y_{j,2}$ for some $j\in\{1,\ldots, n_2-1\}$. The equation $g_1^\prime (z)=0$ has $n_1-1$ distinct solutions in $(-1,1)$. For each of them $|g_1(z)|>1$ and ${g_2}^\prime(g_1(z))\neq 0$. On the other hand, for each $j\leq n_2-1$, the equation $g_1(z)=y_{j,2}$ has $n_1$ distinct solutions with ${g_1}^\prime(y_{j,2})\neq 0$. Thus, the total number of solutions for the equation ${g_3}^\prime(z)=0$ is $n_1-1+n_1(n_2-1)=n_1n_2-1$ which is required. Hence, $g_3$ is admissible. It is straightforward that for the function $g_3$ parts $(a)$ and $(b)$ are satisfied. The part $(c)$ is also satisfied for $g_3$, since arbitrary compositions of even and odd functions are either even or odd.
	\end{proof}
	\begin{lemma}\label{adm2}
		Let $(f_n)\in\mathcal{R}$ be a sequence of admissible polynomials satisfying  $(A).$
		Then $F_n$ is an admissible polynomial with the property  $(A).$
		Besides, $F_{n+1}^{-1}([-1,1])\subset F_{n}^{-1}([-1,1])\subset [-1,1]$ and
		$K=\cap_{n=1}^\infty F_{n}^{-1}([-1,1])$ is a Cantor set in $[-1,1]$.
	\end{lemma}
	\begin{proof}
		All statements except the last one follow directly from Lemma \ref{admlem} and the representation $F_n(z)=f_n\circ F_{n-1}(z)$. Let us show that $K$ is totally disconnected.
		
		If $K$ is polar then (see e.g. Corollary 3.8.5. of \cite{Ransford}) it is totally disconnected.
		If $K$ is non-polar, then (see e.g. Theorem A.16. of \cite{simon1}), $\mu_{F_{n}^{-1}([-1,1])}\rightarrow \mu_K$. Suppose that $K$ is not totally disconnected. Then $K$ contains an interval $E$ such that $E\subset F_{n}^{-1}([-1,1])$ for all $n$. Since we have $\mu_{F_{n}^{-1}([-1,1])}(E)\leq 1/(d_1\ldots d_n)$ by Theorem \ref{inv}, convergence of $(\mu_{F_{n}^{-1}([-1,1])})$ implies that $\mu_K(E)=0$. Thus all interior points of $E$ in $\mathbb{R}$ are outside of the support of $\mu_K$. This is impossible by Theorem 4.2.3. of \cite{Ransford} since $K=\partial(\overline{\mathbb{C}}\setminus K)$ and $\mathrm{Cap}(E)>0$. 
	\end{proof}
	
	Here we consider admissible polynomials as polynomials of complex variable.
	\begin{lemma}\label{eps}
		Let $f$ be an admissible polynomial satisfying  $(A).$ Then $|f(z)|>1+2\epsilon$ provided  $|z|>1+\epsilon$ for $\epsilon>0$. If $|z|=1$ then $|f(z)|>1$ unless $z=\pm 1$.
	\end{lemma}
	\begin{proof}
		Let $\deg f= n$ and $x_1<x_2<\ldots<x_n$ be the zeros of $f$. By $(c),\, x_k=-x_{n+1-k}$ for $k\leq n.$
		In particular, if $n$ is odd, then $x_{(n+1)/2}=0.$
		
		Let  $x_i\neq 0$  and  $\epsilon>0$. Then, by the law of cosines, the polynomial $P_{x_i}(z):=z^2-x_i^2$
		attains minimum of its modulus on the set $\{z: |z|=1+\epsilon\}$ at the point $z=1+\epsilon.$
		Therefore $|P_{x_i}(z)|/|P_{x_i}(\pm 1)|>1+2\epsilon$ for any $z$ with $|z|=1+\epsilon$. Using the symmetry of the roots of $f$ about $x=0$, we see that $|f(z)|=|f(z)/|f(\pm 1)|>1+2\epsilon$ for such $z$.
		
		If $|z|=1$ then $|P_{x_i}(z)|$ attains its minimum  at the points $\pm 1.$ Hence we have $|f(z)|=|f(z)|/|f(\pm 1)|>1$ if $|z|=1$ and $z\neq \pm 1$. 
	\end{proof}
	
	In the next theorem we use the argument of Theorem 1 in \cite{gonc}.
	\begin{theorem}\label{zek}
		Let $(f_n)\in\mathcal{R}$ be a sequence of admissible polynomials satisfying $(A)$. Then $K=\cap_{n=1}^\infty F_n^{-1}([-1,1])=J_{(f_n)}$.
	\end{theorem}
	\begin{proof}
		Let us prove first the inclusion $J_{(f_n)}\subset K$. Let $R>1$ be any number satisfying $A_1 R(1-(A_2/{(R-1)}))>2$. Then by part $(b)$ of Theorem \ref{general}, we have $\mathcal{A}_{(f_n)}(\infty)=\cup_{k=1}^{\infty}{F_k}^{-1}(\triangle_R)$ and $f_n({\overline{\triangle}_R})\subset \triangle_R$ for all $n$. If we show that $|F_n(z)|>1+\epsilon$ for some $n\in\mathbb{N}$ and for some positive $\epsilon$, this implies that $F_{n+k}(z)\in\triangle_R$ for some positive $k$ by Lemma \ref{eps} and thus $z\not\in J_{(f_n)}$.
		
		Let $|z|=1+\epsilon$ where $\epsilon>0$. Then by Lemma \ref{eps}, $|F_1(z)|>1+2\epsilon$. Hence, $z\not\in J_{(f_n)}$.
		
		Let $|z|=1$ where $z\neq\pm 1$. Then using Lemma \ref{eps}, we see that $|F_1(z)|>1$. Thus, $z\not\in J_{(f_n)}$.
		
		If we let $z\in[-1,1]\setminus K$, then there exists a number $N\in\mathbb{N}$ such that $|F_N(z)|>1$. As a result, $z\not\in J_{(f_n)}$.
		
		Letting $z=x+iy$ where $x\not\in K$, $|y|>0$ and $|z|<1$ implies that there exists a positive number $N$ such that $|F_N(x)|>1$. Since all of the zeros of $F_n$ are on the real line by Lemma \ref{adm2}, we have $|F_n(z)|>|F_n(x)|>1$. Hence $z\not\in J_{(f_n)}$.
		
		Let $z=x+iy$ where $x\in K$, $|y|>0$ and $|z|<1$. Since $K$ is a Cantor set by Lemma \ref{adm2}, there exists a number $N\in\mathbb{N}$ such that $n>N$ implies that each connected component of $F_n^{-1}([-1,1])$ has length less than $y^2/8$. Let $x_1<x_2\ldots <x_{d_1\ldots d_{N+1}}$ be the roots of the polynomial $F_{N+1}$ and $E_j$ denote the connected component of $F_{N+1}^{-1}([-1,1])$ containing $x_j$ for $j=1,2,\ldots,d_1\ldots d_{N+1}$. Furthermore, let $E_{s}=[a, b]$ be the component containing the point $x$. Observe that $|F_{N+1}(a)|= |F_{N+1}(b)|=1$. So, in order to show $z\not\in J_{(f_n)}$, it is enough to show that $|F_{N+1}(z)|>|F_{N+1}(a)|$.
		
		If $j<s$, then $|a-x_j|\leq |x-x_j|<|z-x_j|$.
		
		If $j=s$, then $|a-x_j|<y^2/8<|y|\leq |z-x_j|$.
		
		If $j>s$, then 
		\begin{eqnarray*}
			|a-x_j|&=&\sqrt{|x_j-a|^2}\\
			&\leq&\sqrt{|x_j-x|^2+|x-a|^2+2|x_j-x||x-a|}\\
			&<&\sqrt{|x_j-x|^2+\frac{y^4}{64}+\frac{y^2}{2}}\\
			&<&\sqrt{|x_j-x|^2+y^2}= |z-x_j|.
		\end{eqnarray*}
		Therefore, $|F_n(z)|>1$. Thus, we have $J_{(f_n)}\subset K$ and $\overline{\mathbb{C}}\setminus K\subset\mathcal{A}_{(f_n)}(\infty)$ .
		
		For the inverse inclusion, observe that $K\subset \{z: |F_n(z)|\leq 1\mbox{ for all $n$}\}$ where $\{z: |F_n(z)|\leq 1\mbox{ for all $n$}\}\cap \mathcal{A}_{(f_n)}(\infty)=\emptyset$. Since $K$ is contained in the real line and $\overline{\mathbb{C}}\setminus K\subset \mathcal{A}_{(f_n)}(\infty)$ by the first part of the proof, we have $K\subset \partial \mathcal{A}_{(f_n)}(\infty)=J_{(f_n)}$. 
	\end{proof}
	
	\begin{corollary} Orthogonal polynomials associated to the equilibrium measure of $K$ and the corresponding recurrence coefficients (Jacobi coefficients)  can be calculated by Theorem \ref{orti} and Theorem \ref{Jacobi}.
	\end{corollary}
	
	\section{Smoothness of Green's functions}
	For some generalized Julia sets a deeper analysis can be done. In this section we consider a modification
	$K_1(\gamma)$ of the set $K(\gamma)$ from \cite{gonc} that will quite correspond to Theorem \ref{zek}.
	We give a necessary and sufficient condition on the parameters that makes the Green
	function $G_{\overline{\mathbb{C}}\setminus K_1(\gamma)}$ optimally smooth. Although smoothness properties of
	Green functions are interesting in their own rights, in our case the optimal smoothness of
	$G_{\overline{\mathbb{C}}\setminus K_1(\gamma)}$ is necessary for $K_1(\gamma)$ to be a Parreau-Widom set.
	
	Let $K\subset \mathbb{C}$ be a non-polar compact set. Then $G_{\overline{\mathbb{C}}\setminus K}$ is said to be \emph{H\"{o}lder continuous} with exponent $\beta$ if there exists a number $A>0$ such that $$G_{\overline{\mathbb{C}}\setminus K}(z)\leq A (\mathrm{dist}(z,K))^\beta,$$ holds for all $z$ satisfying $\mathrm{dist}(z,K)\leq 1,$ where $\mathrm{dist}(\cdot)$ stands for the distance function. For applications of smoothness of Green functions, we refer the reader to  \cite{Ciez}.
	
	Smoothness properties of Green functions are examined for a variety of sets. For the complement of autonomous Julia sets, see \cite{Kosek} and for the complement of $J_{(f_n)}$ see \cite{Bruck1,Bruck}. When $K$ is a symmetric
	Cantor-type set in $[0,1]$, it is possible to give a sufficient and necessary condition in order the Green function for the complement of the Cantor set is H\"{o}lder continuous with the exponent $1/2$, i.e. optimally smooth. See Chapter $5$ in \cite{tot} for details.
	
	We will use density properties of equilibrium measures. By the next theorem, which is proven in \cite{Tookos}, it is possible to associate the density properties of equilibrium measures with the smoothness properties of Green's functions.
	\begin{theorem}\label{density} Let $K\subset\mathbb{C}$ be a non-polar compact set which is regular with respect to the Dirichlet problem. Let $z_0\in\partial \Omega$ where $\Omega$ is the unbounded component of $\overline{\mathbb{C}}\setminus K$. Then for every  $0<r<1$ we have
		$$\int\limits_{0}^{r}\frac{\mu_K(D_t(z_0))}{t}dt\leq \sup_{|z-z_0|=r} G_{\Omega}(z)\leq3\int\limits_{0}^{4r}\frac{\mu_K(D_t(z_0))}{t}dt.$$
	\end{theorem}

	Let $\gamma:=(\gamma_n)_{n=1}^\infty$ be given such that $0<\gamma_n<1/4$ for all $n,\,\epsilon_n:=1/4-\gamma_n$.
	Take $f_n(z)=\frac{1}{2\gamma_n}(z^2-1)+1$ for $n\in \mathbb{N}.$ Thus, $F_1(z)=\frac{1}{2\gamma_1}(z^2-1)+1$
	and similarly $F_n(z)=\frac{1}{2\gamma_n}(F_{n-1}^2(z)-1)+1$ for $n\geq 2.$ It is easy to see that, as a polynomial of
	real variable, $F_n$ is admissible, it satisfies $(A)$ and, in addition, all minimums of $F_n$ are the same
	and equal to $1-\frac{1}{2\gamma_n}.$ Then $K_1(\gamma)= \cap_{n=1}^\infty F_n^{-1}([-1,1])$ is a stretched version
	of the set $K(\gamma)$ from \cite{gonc}. Here,
	$$G_{\overline{\mathbb{C}}\setminus K_1(\gamma)}(z)=\lim_{n\to \infty} 2^{-n}\,\log|F_n(z)|.$$ Since the leading
	coefficient of $F_n$ is $2^{1-2^n}\gamma_n\,\gamma_{n-1}^2 \cdots \gamma_1^{2^{n-1}},$ the logarithmic capacity
	of $K_1(\gamma)$ is $2 \exp(\sum_{n=1}^{\infty}2^{-n}\log \gamma_{n}).$ 
	
	If, in addition, for some $0<c<1/4$ we have $\gamma_n\geq c$ for all $n$, then  $(f_n)\in\mathcal{R}$ and, by
	Theorem \ref{zek}, $K_1(\gamma)= J_{(f_n)}$. Without this condition the sequence $(f_n)$ is not regular, the set $K_1(\gamma)$ is not uniformly perfect (at least if we assume that $\gamma_n\leq 1/32$ for all $n\in\mathbb{N}$, see Theorem 3 in \cite{gonc}), but polynomials from Theorem \ref{orti} are still orthogonal, by \cite{alpgon}.
	
	In the limit case, when all $\gamma_n=1/4,\,F_n$ is the Chebyshev polynomial (of the first kind) $T_{2^n}$
	and $K_1(\gamma)=[-1,1].$

	Let $I_{1,0}:=[-1,1]$. The set $F_n^{-1}([-1,1])$ is a disjoint union of $2^n$ non-degenerate closed intervals $I_{j,n}=[a_{j,n},b_{j,n}]$ with length $l_{j,n}$ for $1\leq j\leq 2^n.$ We call them {\it basic intervals of $n-$th level}. The inclusion $F_{n+1}^{-1}([-1,1])\subset F_{n}^{-1}([-1,1])$  implies that
	$I_{2j-1,n+1}\cup I_{2j,n+1}\subset I_{j,n}$ where $a_{2j-1,n+1}=a_{j,n}$ and $b_{2j,n+1}=b_{j,n}$.
	We denote the gap $(b_{2j-1,n+1},a_{2j,n+1})$ by $H_{j,n}$ and the length of the gap by $h_{j,n}$. Thus, $$K_1(\gamma)=[-1,1]\setminus \left(\bigcup_{\substack{n=0}}^\infty\bigcup_{\substack{1\leq j\leq 2^n}} H_{j,n}\right).$$

	Let us consider the parameter function $v_{\gamma}(t)=\sqrt{1-2\gamma(1-t)}$  for $|t|\leq 1$ with
	$0<\gamma\leq 1/4.$ This increasing and concave function is an analog of $u$ from \cite{gonc}. By means
	of $v_{\gamma}$ we can write the endpoints of the basic intervals of $n-$th level, which are the solutions
	of $F_k(x)=-1$ for $1\leq k \leq n$ together with the points $\pm 1.$ Namely, $F_n(x)=-1$ gives
	$F_{n-1}(x)=\pm v_{\gamma_n}(-1),$ then $F_{n-2}(x)=\pm v_{\gamma_{n-1}}(\pm v_{\gamma_n}(-1)),$ etc.
	The iterates eventually give $2^n$ values
	\begin{equation}\label{pm}
		x=\pm v_{\gamma_{1}}\circ(\pm v_{\gamma_2}\circ( \cdots \pm v_{\gamma_{n-1}}\circ (\pm v_{\gamma_n}(-1)\cdots),
	\end{equation}
	which are the endpoints $\{b_{2j-1,n},a_{2j,n}\}_{j=1}^{2^{n-1}}.$ The remaining $2^n$ points can be found
	similarly, as the solutions of $F_k(x)=-1$ for $1\leq k <n$ and $\pm 1.$
	
	As in Lemma 2 in  \cite{gonc}, $\min_{1\leq j\leq 2^n}l_{j,n}$ is realized on the first and the last
	intervals. Since the rightmost solution of  $F_n(x)=-1,$ namely $a_{2^n,n},$  is given by \eqref{pm}
	with all signs positive, we have
	\begin{equation}\label{first}
		l_{1,n}=l_{2^n,n}=1-v_{\gamma_{1}}(v_{\gamma_2}( \cdots  v_{\gamma_{n-1}}( v_{\gamma_n}(-1)\cdots).
	\end{equation}
	
	The next lemma shows that $l_{1,n}$ can be evaluated in terms of $\delta_n:=\gamma_1\gamma_2\cdots\gamma_n.$
	
	\begin{lemma}\label{elel}
		For each $\gamma$ with $0<\gamma_k \leq 1/4$ and for all $n\in\mathbb{N}$  we have
		$$2\,\delta_n\leq l_{1,n}\leq(\pi^2/2)\, \delta_n.$$
	\end{lemma}
	
	\begin{proof}
		Clearly, $1-v_{\gamma}(t)= \frac{2}{1+v_{\gamma} (t)}\,\gamma (1-t).$ Repeated application of this to
		\eqref{first} gives the representation $ l_{1,n}=2\, \varkappa_n(\gamma)\,\delta_n,$
		where $\varkappa_n(\gamma)$ is equal to
		$$ \frac{2}{1+v_{\gamma_1}( v_{\gamma_2}(\cdots v_{\gamma_n}(-1)\cdots)}
		\,\,\,\frac{2}{1+v_{\gamma_2}( \cdots v_{\gamma_n}(-1)\cdots)} \cdots
		\frac{2}{1+v_{\gamma_n}(-1)}. $$	Since $ v_{1/4}(t) \leq v_{\gamma}(t)\leq 1,$ we have $1\leq \varkappa_n(\gamma)\leq \varkappa_n(1/4),$	where the last denotes the value of $\varkappa_n$  in the case when all $\gamma_k=1/4.$ This gives the left part of the inequality. Let $C_{2^n}$ be the distance between $1$ and the rightmost extrema of $T_{2^n}.$  Hence, see e.g. p.7. of \cite{rivlin},
		$C_{2^n}=1-\cos(\pi/2^n)<\pi^2/(2\cdot 4^n).$ On the other hand,	$C_{2^n}=2\, \varkappa_n(1/4)\,4^{-n}.$ Therefore, $\varkappa_n(1/4) < \pi^2/4,$ and the lemma follows. 
	\end{proof}
	
	For the case $\gamma_n\leq 1/32$ for all $n$, smoothness of the Green's function for $\overline{\mathbb{C}}\setminus K(\gamma)$ and related properties are examined in \cite{gonc2}, \cite{gonc}. The next theorem is complementary to Theorem 1 of \cite{gonc2} and examines the smoothness of the Green function as $\gamma_n\rightarrow 1/4$.
	
	\begin{theorem}
		The function $G_{\overline{\mathbb{C}}\setminus K_1(\gamma)}$ is H\"{o}lder continuous with the exponent $1/2$ if and only if $\sum_{k=1}^\infty \epsilon_k<\infty$.
	\end{theorem}
	
	\begin{proof} Let us assume that $\sum_{k=1}^\infty \epsilon_k<\infty$. Then $\prod_{k=1}^{\infty}(1-4\epsilon_k)=a$ for some $0<a<1, $
		$\delta_n=4^{-n} \prod_{k=1}^n(1-4\epsilon_k)>a\,4^{-n}$ and, by Lemma \ref{elel},
		$2a\cdot 4^{-n}\leq l_{1,n}$ for all $n\in\mathbb{N}.$
		
		Let $z_0$ be an arbitrary point of $K_1(\gamma)$. We claim that
		$\mu_{K_1(\gamma)}(D_t(z_0)) \leq \frac{4\sqrt 2}{\sqrt a}\sqrt{t}$ for all $t>0.$ It is evident for
		$t\geq 1/32,$ as $\mu_{K_1(\gamma)}$ is a probability measure. Let $0<t<1/32.$
		Fix $n$ with $l_{1,n}<t\leq l_{1,n-1}.$ We have $t>2a\cdot 4^{-n}.$
		
		On the other hand,
		$D_t(z_0)$ can contain points from at most $4$ basic intervals of level $n-1$.	Since $\mu_{F_n^{-1}([-1,1])}\rightarrow \mu_{K_1(\gamma)},$ by \cite{simon1}, we have
		$\mu_{K_1(\gamma)}(I_{j,k})=1/2^k$ for all $k\in\mathbb{N}$ and $1\leq j\leq 2^k$.	Therefore, $\mu_{K_1(\gamma)}(D_t(z_0))\leq 2^{3-n}<8 \sqrt{t/2a},$ which is our claim. The optimal smoothness of $G_{\overline{\mathbb{C}}\setminus K_1(\gamma)}$ follows from Theorem \ref{density}.
		
		Conversely, suppose that, on the contrary, $\sum_{k=1}^\infty\epsilon_k=\infty$. This is equivalent to the condition $4^{n}\delta_n\rightarrow 0$ as $n\rightarrow\infty$. Thus, for any $\sigma>0$, there is a number $N$ such that $n>N$ implies that $4^{n}\delta_n<\sigma$. For any $t\leq l_{1,N+1}$, there exists $m\geq N+1$ such that $l_{1,m+1}<t\leq l_{1,m}$. Then, $\mu_{K_1(\gamma)}(D_t(0))\geq \mu_{K_1(\gamma)}(I_{1,m+1})=2^{-m-1}.$ On the other hand, by Lemma \ref{elel},
		$t\leq 2\pi^2 \sigma \,4^{-m-1}$. Therefore, for any $t\leq l_{1,N+1}$ we have
		$\frac{\sqrt t}{ \pi \sqrt {2\,\sigma}}\leq \mu_{K_1(\gamma)}(D_t(0))$. Hence, the inequality
		$$\frac{\sqrt 2}{\pi\sqrt \sigma} \,\sqrt r \leq\int_{0}^r \frac{\mu_{K_1(\gamma)}(D_t(0))}{t}dt,$$
		holds for $r \leq l_{1,N+1}$. By Theorem \ref{density},
		$G_{\overline{\mathbb{C}}\setminus K_1(\gamma)}(-r) \geq \frac{\sqrt 2}{\pi\sqrt \sigma} \,\sqrt r.$
		Since $\sigma$ is  here  as small as we wish, the Green function is not optimally smooth. 
	\end{proof}
	
	\section{Parreau-Widom sets}
	Parreau-Widom sets are of special interest in the recent spectral theory of orthogonal polynomials. For different aspects of the theory, we refer the reader to the articles \cite{christiansen,gesztesy,Peherstorfer,volberg} among others.
	
	A compact set $K\subset\mathbb{R}$ which is regular with respect to the Dirichlet problem is called a \emph{Parreau-Widom} set if $PW(K):=\sum_{j} G_{\overline{\mathbb{C}}\setminus K}(c_j)<\infty$ where $\{c_j\}$ is
	the set of critical points of $G_{\overline{\mathbb{C}}\setminus K},$ which, clearly, is at most countable.
	A Parreau-Widom set has always positive Lebesgue measure, see \cite{christiansen}.
	
	Our aim is to give a criterion when $K_1(\gamma)$ is a Parreau-Widom set.
	Note that, since autonomous Julia-Cantor sets in $\mathbb{R}$ have zero Lebesgue measure (see e.g. Section 1.19. in \cite{Lyubich}), such sets can not be Parreau-Widom.
	
	We begin with a technical lemma.
	
	\begin{lemma}\label{bbb} Given $p\in \mathbb{N},$ let $b_0=1$ and $ b_{k+1}=b_k(1+4^{-p+k}\,b_k)$ for
		$0\leq k \leq p-1.$ Then $b_p <2.$
	\end{lemma}
	\begin{proof} We have $b_1=1+4^{-p},\, b_2=1+(1+4)\,4^{-p}+2\cdot 4\cdot 4^{-2p}+4\cdot 4^{-3p}, \cdots,$
		so $b_k= \sum_{n=0}^{N_k}a_{n,k}  4^{-np} $ with $N_k=2^k-1$ and $a_{0,k}=1$. Let $a_{n,k}:=0$ if $n>N_k$. The definition of
		$b_{k+1}$ gives the recurrence relation
		\begin{equation}\label{b}	a_{n,k+1}= a_{n,k}+ 4^k\sum_{j=1}^n a_{n-j,k}\, a_{j-1,k} \,\,\, \mbox {for}\,\,\, 1\leq n \leq N_{k+1}.
		\end{equation}
		If $N_k< n \leq N_{k+1},$ that is $n=N_{k}+m$ with  $1\leq m \leq N_k+1,$ then the formula takes the form
		$a_{n,k+1}= 4^k\sum_{j=m}^{n-m+1} a_{n-j,k}\, a_{j-1,k},$ since $ a_{n-j,k}=0$ for $j<m$ and
		$a_{j-1,k}=0$ for $j>n-m+1.$ In particular, $a_{N_{k+1},k+1}= 4^k\,a_{N_k,k}^2$ and
		$a_{1,k+1}= a_{1,k}+ 4^k.$ Therefore, $a_{1,k}=1+4+\cdots +4^{k-1}< 4^k/3.$
		Let us show that $a_{n,k}< C_n\,4^{nk}$ with $C_n=4^{1-n}/3$ for $n \geq 2.$ This gives the desired result,
		as $b_p=\sum_{n=0}^{N_p}a_{n,p} 4^{-np}< 1+ 1/3 \cdot \sum_{n=1}^{N_p}4^{1-n}<2.$
		
		By induction, suppose the inequality $a_{j,k}< C_j\,4^{jk}$ is valid for $1\leq j \leq n-1$ and for all $k>0.$
		We consider $j=n.$ The bound $a_{n,i}< C_n\,4^{ni}$ is valid for $i=1,$ as $a_{n,1}=0$ for $n\geq 2.$ Suppose it is valid as well
		for $i\leq k.$
		
		We use \eqref{b} repeatedly, in order to reduce the second index, and, after this, the induction hypothesis:
		\begin{align*} a_{n,k+1}&= \sum_{q=1}^k 4^q\sum_{j=1}^n a_{n-j,q}\, a_{j-1,q}<\sum_{q=1}^k 4^{nq}\sum_{j=1}^n C_{n-j}\, C_{j-1}
			< \sum_{q=1}^k 4^{nq}\\ 
			&< C_n \,4^{n(k+1)}, 
		\end{align*}
		where $C_0:=1.$ Therefore the desired bound is valid for
		all positive $n$ and $k$. 
	\end{proof}

	\begin{theorem}$K_1(\gamma)$ is a Parreau-Widom set if and only if
		$\,\,\sum_{k=1}^\infty \sqrt{\epsilon_k} <\infty$.
	\end{theorem}
	\begin{proof}
		Let $E_n=\{z\in \mathbb{C}: |F_n(z)|\leq 1\}.$ Then
		$G_{\overline{\mathbb{C}}\setminus E_n }(z)=2^{-n}\,\log |F_n(z)|.$ Clearly,  the critical points of
		$G_{\overline{\mathbb{C}}\setminus E_n }$ coincide with the critical points of
		$F_n$ and thus they are real. Let $Y_n=\{x: F_n^{\,'}(x)=0\},\,\, Z_n=\{x: F_n(x)=0\}.$ Clearly, $Y_n\cap Z_n=\emptyset$ and $Z_k\cap Z_n=\emptyset$
		for $n\ne k.$ Since $ F_n^{\,'}= F_{n-1} F_{n-1}^{\,'}/\gamma_n,$ we have $Y_n=Y_{n-1} \cup Z_{n-1},$
		so $Y_n=Z_{n-1} \cup Z_{n-2} \cup  \cdots \cup Z_0,$ where $Z_0=\{0\}.$ We see that
		$Y_n \subset Y_{n+1},$ so the set of critical points for $G_{\overline{\mathbb{C}}\setminus K_1(\gamma)}$
		is $\cup_{n=0}^{\infty}Z_n$ and
		$ PW(K_1(\gamma))=  \sum_{n=1}^{\infty}\sum_{z\in Z_{n-1}}G_{\overline{\mathbb{C}}\setminus K_1(\gamma)}(z).$
		In addition, for each $k\geq n$ the function $F_k$ is constant on the set $ Z_{n-1}$ which contains
		$2^{n-1}$ points. Let $s_n=2^{n-1} G_{\overline{\mathbb{C}}\setminus K_1(\gamma)}(z),$ where $z$ is any point from $Z_{n-1}$. Then
		\begin{equation}\label{PW}
			PW(K_1(\gamma))=\sum_{k=1}^{\infty} s_k.
		\end{equation}
		
		We can assume that $\sum_{k=1}^{\infty} \epsilon_k < \infty.$ Indeed, it is immediate if
		$\sum_{k=1}^\infty \sqrt{\epsilon_k} <\infty$. On the other hand, if $z\in Z_{n-1},$ that is $F_{n-1}=0,$
		then $F_n(z)=1-1/2\gamma_n=-1-\frac{8\epsilon_n}{1-4 \epsilon_n}.$ Since	$G_{\overline{\mathbb{C}}\setminus E_n } \nearrow G_{\overline{\mathbb{C}}\setminus K_1(\gamma)},$  we have	$ s_n > 1/2 \,\log|F_n(z)| > 1/2 \,\log (1+ 8\epsilon_n)>2\epsilon_n,$ as  $\log (1+ t)>t/2$ for $0<t<2.$
		Therefore the supposition  $PW(K_1(\gamma)) < \infty$ implies, by  \eqref{PW}, that
		$\sum_{k=1}^{\infty} \epsilon_k < \infty.$
		
		Let $a=\prod_{k=1}^{\infty}(1-4 \epsilon_k).$ By the remark above, $0<a<1.$ Our  aim is to evaluate $s_n$
		from both sides for large $n$. Let us fix $N\in\mathbb{N}$ such that $n>N$ implies that  $\epsilon_n \leq a/36.$ We consider only such $n$ after this point of the proof. Then $1-4 \epsilon_n >8/9$ and
		for $\sigma_n:=\frac{8\epsilon_n}{1-4 \epsilon_n}$ there exists $p \in \mathbb{N}$ such that
		\begin{equation}\label{def p}
			a\cdot 4^{-1-p} < \sigma_n \leq a\cdot 4^{-p}.
		\end{equation}
		
		Consider the function $f(t)=\frac{1}{2\beta}(t^2-1)+1$ for $t>1,$ where $\beta =1/4-\epsilon$
		with $\epsilon < 1/36.$ Thus, $F_{k+1}(z)=f(F_k(z))$ for $ \beta= \gamma_{k+1}.$
		If $t=1+\sigma$ for small $\sigma,$ then we will use the representation
		$f(t)= 1 + \sigma_1$ with $4 \sigma < \sigma_1 = 4 \sigma \,\,\frac{1+\sigma/2}{1-4\epsilon}.$
		Also, for each $t\geq 1$ we have $t^2 \leq f(t) < \frac{1}{2\beta}\,t^2 < \frac{9}{4}\,t^2.$
		
		Let us fix $z\in  Z_{n-1}$. Then, as above, $|F_n(z)|=1+\sigma_n.$ Then $ F_{n+1}(z)=1+\sigma_{n+1}$
		with $4 \sigma_n < \sigma_{n+1} = 4 \sigma_{n} \,\,\frac{1+\sigma_n/2}{1-4\epsilon_{n+1}}.$
		We continue in this fashion to obtain  $ F_{n+p}(z)=1+\sigma_{n+p}$ with
		
		\begin{equation}\label{n+p}
			4^p \, \sigma_n <  \sigma_{n+p}= 4^p \, \sigma_n \cdot \prod_{k=n}^{n+p-1} \frac{1+\sigma_k/2}{1-4\epsilon_{k+1}}.
		\end{equation}
		
		After that we use the second estimation for $f.$ This gives
		$ F_{n+p}^2(z)\leq F_{n+p+1}(z) < \frac{9}{4}\, F_{n+p}^2(z)$ and, for each $k \in \mathbb{N},$
		$$  F_{n+p}^{2^k}(z)\leq F_{n+p+k}(z) < (9/4)^{2^k-1} \, F_{n+p}^{2^k}(z).$$
		From this, we have  $$ 2^{-n-p} \,\log F_{n+p}(z) \leq G_{\overline{\mathbb{C}}\setminus E_{n+p+k} }(z)
		\leq  2^{-n-p} [ \log (9/4)+ \log F_{n+p}(z)].$$ Recall that  $$G_{\overline{\mathbb{C}}\setminus E_{n+p+k} }(z)  \nearrow G_{\overline{\mathbb{C}}\setminus K_1(\gamma)}(z),$$ as $k \to \infty$
		and  $s_n=2^{n-1} G_{\overline{\mathbb{C}}\setminus K_1(\gamma)}(z).$ Hence,
		$$ 2^{-p-1} \,\log F_{n+p}(z) \leq s_n \leq  2^{-p-1} [ \log (9/4)+ \log F_{n+p}(z)].$$
		
		Now suppose that $K_1(\gamma)$ is a Parreau-Widom set, so, by \eqref{PW}, the series $\sum_{k=1}^{\infty} s_k$
		converges. Then, by \eqref{n+p}, we have $s_n \geq  2^{-p-1} \,\log (1+ 4^p \, \sigma_n).$
		By \eqref{def p}, $4^p \, \sigma_n<1$ and $\log (1+ 4^p \, \sigma_n)> 4^p \, \sigma_n/2.$ Therefore,
		$s_n \geq  2^p \, \sigma_n/4.$ We use  \eqref{def p} once again to obtain $s_n \geq \sqrt{a\,\sigma_n}/8,$
		which implies the convergence of $\sum_{k=1}^{\infty} \sqrt{\epsilon_k}.$\\
		
		Conversely, suppose that $\,\sum_{k=1}^\infty \sqrt{\epsilon_k} <\infty$. Then
		$s_n \leq 2^{-p} \log(3/2) + 2^{-p-1}\sigma_{n+p}.$  By \eqref{def p}, the first summand on the right is the general term of a
		convergent series. For the addend we have
		$$2^{-p-1}\sigma_{n+p} <  1/2a \cdot 2^{p}\,\sigma_n \,\prod_{k=n}^{n+p-1} (1+\sigma_k/2),$$ by \eqref{n+p}.
		From \eqref{def p} it follows that $2^{p}\,\sigma_n \leq \sqrt{a \sigma_n}< 3 \sqrt{a \epsilon_n},$ as
		$\epsilon_n<1/36.$ Let us show that
		\begin{equation}\label{prod}
			\prod_{k=n}^{n+p-1} (1+\sigma_k/2)<2.
		\end{equation}
		This will give the estimation
		$2^{-p-1}\sigma_{n+p} < 3 \sqrt{\epsilon_n/a},$ where the right part is the general term of a convergent series. Then $\sum_{k=1}^{\infty}s_k<\infty, $ which is the desired conclusion, by \eqref{PW}.
		
		Thus, it remains to prove  \eqref{prod}. We use notations of Lemma \ref{bbb}. By  \eqref{def p}, we have
		$1+\sigma_n/2 \leq 1+\,a\,4^{-p}/2 < b_1.$ Then,
		$$1+\sigma_{n+1}/2 < 1+\,\frac{a}{1-4\epsilon_{n+1}}\,4^{-p+1}(1+\sigma_n/2)<1+4^{-p+1}\,b_1=b_2/b_1$$
		and $(1+\sigma_n/2 )( 1+\sigma_{n+1}/2) <b_2.$ Similarly, by  \eqref{n+p} and  \eqref{def p},
		$$ 1+\sigma_{n+k+1}/2 < 1+\,\frac{a}{(1-4\epsilon_{n+1})\cdots (1-4\epsilon_{n+k})} \,\,4^{-p+k}\, b_k < b_{k+1}/b_k$$
		for $k\leq p-2.$ Lemma \ref{bbb} now yields \eqref{prod}. 
	\end{proof}

	\bibliographystyle{spmpsci}

\end{document}